\documentclass[reqno]{amsart}
\usepackage{hyperref}
\usepackage{color} 
\usepackage{comment}

\usepackage{amsfonts,amsthm,amsmath}
\usepackage[shortlabels]{enumitem}
\setlist[enumerate,1]{label={\upshape(\roman*)}}

\usepackage{amssymb}

\theoremstyle{plain}
\newtheorem{thm}{Theorem}[section]
\newtheorem{lem}[thm]{Lemma}

\theoremstyle{definition}
\newtheorem{df}[thm]{Definition}
\newtheorem*{rem*}{Remark}

\newcommand{\Frame}{\mathcal{F}(n)}
\newcommand{\Window}{\mathcal{W}(n)}

\title{Spanning trees in directed square cycles}

\author[Tanaka]{Yuuho Tanaka}
\address{Faculty of Science and Technology, Oita University, Oita, 870-1192, Japan}
\email{tanaka-yuuho@oita-u.ac.jp}

\begin{document}


\begin{abstract}
We classify weakly connected spanning closed (WCSC) subgraphs of $\overrightarrow{C_n^2}$, the square of a directed $n$-vertex cycle.
Then we show that every spanning tree of $\overrightarrow{C_n^2}$ is contained in a unique nontrivial WCSC subgraph of $\overrightarrow{C_n^2}$.
As a result, we obtain a purely combinatorial derivation of the formula for the number of directed spanning trees of $\overrightarrow{C_n^2}$.
Moreover, we obtain the formula for the number of directed spanning trees of $\overrightarrow{C_n^2}$, which is a Jacobsthal number.
\end{abstract}
\maketitle

\section{Introduction}

Networks, such as those in social science and computing, can be represented as graphs.
Therefore, graphs are often used to analyze the complexity of networks.
The spanning trees of a graph correspond to paths that connect all vertices in a network, and such trees are important indicators for characterizing the complexity and reliability of the network~\cite{MY}.

In general, the number of spanning trees of a connected graph can be counted using the Matrix Tree Theorem~\cite{GR, Lee, Tutte}.
More precisely, the number of spanning trees in a graph is given by the product of the non-zero eigenvalues of its Laplacian matrix divided by the number of vertices in the graph.
Explicit formulas for the number of spanning trees have been derived for undirected circulant graphs~\cite{Louis, Med2, Med, Yong}, and for directed graphs, explicit formulas for the number of spanning trees have been obtained for directed circulant graphs and similar structures~\cite{Chen, Louis, WF}.

In particular, for the square of an undirected $n$-vertex cycle, Baron et~al.~\cite{Baron} derived the number of spanning trees using the Matrix Tree Theorem. 
This method allows the derivation of formulas for the number of spanning trees regardless of whether the number of vertices is even or odd.
Kleitman et~al.~\cite{Kleitman} counted the number of spanning trees in the square of an undirected cycle graph without using the Matrix Tree Theorem.
This alternative method provides additional insights into the structure of spanning trees in the square of a cycle graph.
Specifically, they counted the spanning trees using the topological properties of planar graphs only when the number of vertices was even, and they expressed the number of spanning trees in terms of the Fibonacci sequence.
However, for the case where the number of vertices was odd, they only mentioned that the spanning trees could be counted in a similar manner without explicitly deriving a formula.  
Munemasa et~al.~\cite{MT} successfully extended this method to count the number of spanning trees in the square of an undirected cycle graph, even when the number of vertices is odd, without using the Matrix Tree Theorem.
Furthermore, they suggested that a similar approach could be applied to higher powers of cycle graphs, such as the cube of an undirected cycle graph.  

On the other hand, Wojciechowski and Fellows~\cite{WF} derived the number of directed spanning trees for the square of a directed $n$-vertex cycle using the Matrix Tree Theorem.

\begin{thm}[Wojciechowski and Fellows~\cite{WF}]
The number of the directed spanning trees rooted by a vertex of the square of a directed $n$-vertex cycle is
\[
\frac{1}{n}\prod_{k=1}^{n-1}\left(2-\cos \frac{2k\pi}{n}-\sqrt{-1}\sin\frac{2k\pi}{n}-\cos \frac{4k\pi}{n}-\sqrt{-1}\sin\frac{4k\pi}{n}\right).
\]
\end{thm}

In this paper, we count the number of spanning trees in the square of a directed cycle graph without using the Matrix Tree Theorem as in \cite{MT}.
Instead, we utilize the fact that the directed spanning trees of a directed strip graph correspond to a special class of directed spanning trees.
Moreover, we express the number of spanning trees in the square of the directed cycle graph using the Jacobsthal sequence.

The key idea in our proof is the fact that every directed spanning tree of $\overrightarrow{C_n^2}$ is contained in a unique nontrivial weakly connected spanning closed (WCSC) subgraph.

This paper is organized as follows.
In Section~\ref{sec:2}, we establish notation for the directed square cycle as a circulant graph, and we give some properties of the Jacobsthal sequence.
In Section~\ref{sec:3}, we give a classification of WCSC subgraphs of $\overrightarrow{C_n^2}$.
In Section~\ref{sec:4}, we show that the set of the directed spanning trees of $\overrightarrow{C_n^2}$ coincides with the disjoint union of the set of the directed spanning trees of directed strip graphs with tails $S_{n,k,j}$, defined in Section~\ref{sec:2}.
As a consequence, we deduce a combinatorial proof of the formula for $t(\overrightarrow{C_n^2}
)$.

\section{Preliminaries}\label{sec:2}

\begin{df}[Harary~\cite{Harary}]
Let $\overrightarrow{G}$ be a directed graph and
\[
E=\{\{u,v\} \mid (u,v)\in E(\overrightarrow{G}),\;u,v\in V(\overrightarrow{G})\}.
\]
If $G=(V(\overrightarrow{G}),E)$ is connected, then $\overrightarrow{G}$ is calld a \emph{weakly connected graph}.

For a directed graph $\overrightarrow{G}$, we say that $\overrightarrow{G'}$ satisfying
\[
E(\overrightarrow{G'})\subseteq E(\overrightarrow{G}),V(\overrightarrow{G})=V(\overrightarrow{G'})
\]
is a \emph{spanning subgraph} of $\overrightarrow{G}$.
If a spanning subgraph $\overrightarrow{G'}$ in a weakly connected graph $\overrightarrow{G}$ satisfies the following, then $\overrightarrow{G'}$ is called a \emph{directed spanning tree rooted by $u$} of the graph $G$:
\begin{itemize}
\item $\overrightarrow{G'}$ is a weakly connected graph;
\item there are no closed paths in $\overrightarrow{G'}$;
\item for any $v\in V(\overrightarrow{G'})\setminus\{u\}$, its in-degree is $1$;
\item for any $v\in V(\overrightarrow{G'})\setminus\{u\}$, there exists some directed paths from a root $u$.
\end{itemize}
\end{df}

\begin{df}
Let $n$ be an integer with $n\geq5$.
The \emph{square of the directed $n$-vertex cycle},
or the \emph{directed square cycle} for short,
denoted  $\overrightarrow{C_n^2}$,
is defined by $V(\overrightarrow{C^2_n})= \mathbb{Z}_n=\mathbb{Z}/n\mathbb{Z}$, $E(\overrightarrow{C^2_n})=\{(v_i,v_j)\mid v_i,v_j\in V(\overrightarrow{C_n^2}),\;i,j\in\mathbb{Z},\;j-i=1,2\}$, where $v_i=i+n\mathbb{Z}\in\mathbb{Z}_n$.
\end{df}

This is equivalent to a Cayley graph Cay$(\mathbb{Z}_n,\{+1,+2\})$ (see \cite{Tanaka}).

Let $n$ be an integer with $n\geq5$.
Then, $E(\overrightarrow{C_n^2})=\{e_i \mid i\in\mathbb{Z}\}\cup\{f_i\mid i\in\mathbb{Z}\}$, where we define \emph{frame} $e_i$ and \emph{window} $f_i$ as follows:
\[
e_i=(v_i,v_{i+1}),\;f_i=(v_i,v_{i+2}) \quad (i\in\mathbb{Z}).
\]
We denote by $\Frame$ and $\Window$ the sets of frames and windows, respectively, as follows:
\begin{align*}
\Frame&=\{e_i \mid 1\leq i\leq n\}, \\
\Window&=\{f_i\mid 0\leq i\leq n-1\}.
\end{align*}

By a \emph{triangle} of $\overrightarrow{C_n^2}$ (see \cite{GLMY2}) we mean a set
\[
T_i=\{e_i,e_{i+1},f_i\} \quad \text{($i\in\mathbb{Z}$).}
\]
Then,
\[
E(\overrightarrow{C_n^2})=\bigcup_{i=0}^{n-1}T_i.
\]

\begin{df}
Given $i$ ($i\in\mathbb{Z}$), if a subgraph $\overrightarrow{G}$ of $\overrightarrow{C_n^2}$ satisfies $|T_i\cap E(\overrightarrow{G})|\leq1$ or $T_i\subseteq E(\overrightarrow{G})$, then $\overrightarrow{G}$ is said to be \emph{closed} with respect to the triangle $T_i$.
A subgraph $\overrightarrow{G}$ of $\overrightarrow{C_n^2}$ is said to be \emph{closed} if $\overrightarrow{G}$ is closed with respect to $T_i$ for all $i$ ($i\in\mathbb{Z}$).
\end{df}

\begin{df}
The directed graph $\overrightarrow{S_k}$ defined by $V(\overrightarrow{S_k})=\{1,2,\dots,k\}$,
$E(\overrightarrow{S_k})=\{(i,j) \mid i,j\in V(\overrightarrow{S_k}),1\leq j-i\leq 2\}$ 
is called a \emph{directed strip graph} (see \cite{GLMY1}).
\end{df}

The sequence of numbers $J_n$ defined by the recurrence relation $J_0=0,J_1=1,J_{n+2}=J_{n+1}+2J_n$
($n=0,1,2,\dots$) is called the \emph{Jacobsthal sequence}.
The following lemma is due to Horadam~\cite{AFH}.

\begin{lem}
\label{jac3}
For $n\geq2$,
\[
\sum_{k=1}^{n-1}J_i=\frac{1}{2}(J_{n+1}-(-1)^n).
\]
\end{lem}

Moreover, we have the following lemma.
\begin{lem}
\label{lem:1}
Let $t(\overrightarrow{S_n},1)$ be the number of directed spanning trees rooted by $1$ of $\overrightarrow{S_n}$.
For $n\geq2$, 
\[
t(\overrightarrow{S_n},1)=J_{n-2}+J_{n-1}.
\]
\end{lem}

\begin{proof}
We prove the assertion by induction on $n$.
The assertion is trivial if $n=2$, 
we may assume that $n\geq 3$.

Let $t(\overrightarrow{S_n},1;e)$ be the number of directed spanning trees rooted by $1$ in the directed strip graph $\overrightarrow{S_n}$ that contain the edge $e$.
From the definition of a directed spanning tree, any directed spanning tree rooted by $1$ in $\overrightarrow{S_n}$ has an edge of either, $(n-2,n)$ or $(n-1,n)$. 
Thus we have $t(\overrightarrow{S_n},1)=t(\overrightarrow{S_n},1;(n-1,n))+t(\overrightarrow{S_n},1;(n-2,n))$, and 
we have $t(\overrightarrow{S_n},1;(n-1,n))=t(\overrightarrow{S_n},1;(n-2,n))=t(\overrightarrow{S_{n-1}},1)$.
Therefore, we have
\[
t(\overrightarrow{S_n},1)=2\times t(\overrightarrow{S_{n-1}},1)=2(J_{n-3}+J_{n-2})=J_{n-2}+J_{n-1}.
\]
\end{proof}

The following substructures are revised by the substructure in \cite{Kleitman} for directed square cycles.

\begin{df}
Let $n\geq5$. For integers $j$ and $k$ with
$0\leq k\leq \lceil\frac{n-2}{2}\rceil$, we define 
the graph $\overrightarrow{S_{n,k,j}}$ as follows:
\begin{align*}
V(\overrightarrow{S_{n,k,j}})&=V(\overrightarrow{C_n^2}),\\
E(\overrightarrow{S_{n,k,j}})&=E(\overrightarrow{C_n^2})\setminus ES(n,k,j), 
\quad (j,k\in\mathbb{Z},\;0\leq k\leq \lceil\frac{n-2}{2}\rceil),
\end{align*}
where
\[
ES(n,k,j)=\{f_{j-2},f_{j+2k-1}\}\cup\{e_{j-1},\dots,e_{j+2k-1}\} \quad (j,k\in\mathbb{Z},\;0\leq k\leq \lceil\frac{n-2}{2}\rceil).
\]
The graph $\overrightarrow{S_{n,k,j}}$ is called a \emph{directed strip graph with tails}, and $ES(n,k,j)$ is called the \emph{escape route}.
\end{df}

An escape route is the set of edges crossed by a path from the interior to the outside region, in the planar drawing of $\overrightarrow{C_n^2}$. 
Removing the escape route from $\overrightarrow{S_{n,k,j}}$ results in a directed strip graph with tails.

The graphs $\overrightarrow{S_{n,k,j}}$ are WCSC subgraphs of $\overrightarrow{C_n^2}$.
Clearly, $\overrightarrow{C_n^2}$ and $(\mathbb{Z}_n,\{f_0,f_1,\ldots,f_{n-1}\})$ for $n$ odd are also weakly connected spanning subgraphs of $\overrightarrow{C_n^2}$, and we call these subgraphs \emph{trivial} weakly connected spanning subgraphs.

For a graph $\overrightarrow{G}$, let $T_{\overrightarrow{G},u}$ be the set of all directed spanning trees rooted by $u$ of $\overrightarrow{G}$.
Then $t(\overrightarrow{G},u)=|T_{\overrightarrow{G},u}|$.
Since $\overrightarrow{S_{n,k,j}}$ can be obtained from the directed strip graph $\overrightarrow{S_{n-2k}}$ by attaching two tails of length $k$, the following lemma holds.

\begin{lem} \label{lem:strip_tail}
For $v_j\in\mathbb{Z}_n,\;j,k\in\mathbb{Z},\;0\leq k\leq \lceil\frac{n-2}{2}\rceil$,
\[
t(\overrightarrow{S_{n,k,j}},v_j)
=t(\overrightarrow{S_{n-2k}},j)
=t(\overrightarrow{S_{n-2k}},1).
\]
\end{lem}

\section{Classification of spanning subgraphs}\label{sec:3}

In this section, we give a classification of WCSC subgraphs of $\overrightarrow{C_n^2}$.

\begin{lem} \label{lem:pn}
Let $\overrightarrow{G}$ be a WCSC subgraph of $\overrightarrow{C_n^2}$.
If $k$ and $p$ are integers with $0\leq p<n$ and
\[
\{e_{k-1},f_k,f_{k+2},\dots,f_{k+2p-2},e_{k+2p}\}\subseteq E(\overrightarrow{G}),
\]
then $\{e_k,e_{k+1},\dots,e_{k+2p-1}\}\subseteq E(\overrightarrow{G})$.
\end{lem}

\begin{proof}
We prove the assertion by induction on $p$.
The assertion is trivial, if $p=0$, we may assume that $p\geq1$.

Suppose that there exists an integer $i$ with $0\leq i\leq 2p-1$ such that $e_{k+i}\in E(\overrightarrow{G})$.
If $i$ is even, then since $\overrightarrow{G}$ is closed with respect to $T_{k+i}$, we have $e_{k+i+1}\in E(\overrightarrow{G})$.
Therefore, we can apply the induction to 
\[
\{e_{k-1},f_k,f_{k+2},\dots,f_{k+i-2},e_{k+i}\}
\]
and
\[
\{e_{k+i+1},f_{k+i+2},f_{k+i+4},\dots,f_{k+2p-2},e_{k+2p}\}.
\]
Similarly, if $i$ is odd, then 
we can apply the induction. 

It remains to derive a contradiction by assuming
\begin{equation} \label{eq:k3}
e_{k},e_{k+1},\dots,e_{k+2p-1}\notin E(\overrightarrow{G}).
\end{equation}
Since $\overrightarrow{G}$ is closed with respect to $T_{k-1}$, we have
\begin{equation} \label{eq:k1}
f_{k-1}\notin E(\overrightarrow{G}).
\end{equation}
Similarly, since $\overrightarrow{G}$ is closed with respect to
$T_{k+2p-1}$, we have
\begin{equation} \label{eq:k5}
f_{k+2p-1}\notin E(\overrightarrow{G}).
\end{equation}
From 
\eqref{eq:k3}, 
\eqref{eq:k1}, 
and \eqref{eq:k5}, 
we see that the set
$\{v_{k+1},v_{k+3},\dots,v_{k+2p-1}\}$ is separated
from its complement in the weakly connected spanning subgraph $\overrightarrow{G}$. 
This is a  contradiction.
\end{proof}

\begin{lem} \label{lem:2-3}
Let $\overrightarrow{G}$ be a nontrivial WCSC subgraph of $\overrightarrow{C_n^2}$.
If $E(\overrightarrow{G})$ contains no frame, then $n$ is odd, and $\overrightarrow{G}=\overrightarrow{S_{n,\frac{n-1}{2},j}}$ for some integer $j$ with $0\leq j\leq n-1$.
\end{lem}

\begin{proof}
By the assumption, $E(\overrightarrow{G})$ consists only of windows.
Since $\overrightarrow{G}$ is weakly connected, $n$ is odd.
Furthermore, since $\overrightarrow{G}$ is nontrivial, we have
\(
|E(\overrightarrow{G})|\leq n-1.
\)
Additionally, because $\overrightarrow{G}$ is weakly connected, it follows that $|E(\overrightarrow{G})|\geq n-1$.
Thus, we conclude that $|E(\overrightarrow{G})|= n-1$.
Then there exists $j$ such that $E(\overrightarrow{G})=\Window\setminus\{f_j\}=E(\overrightarrow{S_{n,\frac{n-1}{2},j}})$.
This proves $\overrightarrow{G}=\overrightarrow{S_{n,\frac{n-1}{2},j}}$.
\end{proof}

\begin{lem} \label{lem:2-2}
Let $\overrightarrow{G}$ be a nontrivial WCSC subgraph of $\overrightarrow{C_n^2}$.
If $E(\overrightarrow{G})$ contains a frame, then $\overrightarrow{G}=\overrightarrow{S_{n,k,j}}$ for some integers $j,k$ with $0\leq j\leq n-1$, $0\leq k\leq\lfloor\frac{n-2}{2}\rfloor$.
\end{lem}

\begin{proof}
If $\Frame\subset E(\overrightarrow{G})$, then it is easy to see that $\overrightarrow{G}=\overrightarrow{C_n^2}$, contradicting the assumption that $\overrightarrow{G}$ is nontrivial.
Since $\Frame\cap E(\overrightarrow{G})\neq\emptyset$, 
there exists $i,\ell$ ($0\leq i\leq n-1$, $1\leq \ell\leq n-1$) satisfying $\{e_i,e_{i+1},\dots,e_{i+\ell-1}\}\subseteq E(\overrightarrow{G})$ and $e_{i-1},e_{i+\ell}\notin E(\overrightarrow{G})$.
Without loss of generality, we may assume that $i=0$.
In this case, we have
\begin{align}
\{e_0,e_1,\dots,e_{\ell-1}\}&\subseteq E(\overrightarrow{G}),  \label{eq:il1}\\
e_{-1}&\notin E(\overrightarrow{G}),  \label{eq:il1-1}\\
e_\ell&\notin E(\overrightarrow{G}).  \label{eq:il1-12}
\end{align}
Since $\overrightarrow{G}$ is closed with respect to $T_j$ ($0\leq j\leq \ell-2$),
\eqref{eq:il1} implies
\begin{equation} \label{eq:il11}
f_0,f_1,\dots,f_{\ell-2}\in E(\overrightarrow{G}).
\end{equation}
Furthermore, since $\overrightarrow{G}$ is closed with respect to $T_{-1}$, \eqref{eq:il1} and \eqref{eq:il1-1} imply
\begin{equation} \label{eq:il12}
f_{-1}\notin E(\overrightarrow{G}).
\end{equation}
Additionally, because $\overrightarrow{G}$ is closed with respect to $T_{\ell-1}$, \eqref{eq:il1} and \eqref{eq:il1-12} imply
\begin{equation} \label{eq:il31}
f_{\ell-1}\notin E(\overrightarrow{G}).
\end{equation}

Let $s$ and $t$ be the largest non-negative integers such that
\begin{equation} \label{eq:il2}
f_{-2},f_{-4},\dots,f_{-2s}\in E(\overrightarrow{G})
\end{equation}
and
\begin{equation} \label{eq:il3}
f_\ell,f_{\ell+2},\dots,f_{\ell+2t-2}\in E(\overrightarrow{G}),
\end{equation}
respectively.
Then, we have $f_{-2s-2}\notin E(\overrightarrow{G})$ and $f_{\ell+2t}\notin E(\overrightarrow{G})$.
We show that
\begin{equation} \label{eq:tn2}
e_\ell,e_{\ell+1},\dots,e_{\ell+2t}\notin E(\overrightarrow{G})\text{ and }t<\frac{n-\ell}{2},
\end{equation}
\begin{equation} \label{eq:sn2}
e_{-1},e_{-2},\dots,e_{-2s-1}\notin E(\overrightarrow{G})\text{ and }s<\frac{n-\ell}{2}.
\end{equation}
Assume that there exists an integer $m$ with $0\leq m\leq 2t$ and that $e_{\ell+m}\in E(\overrightarrow{G})$;
we may choose minimal such $m$.
By \eqref{eq:il1-12}, we have $m>0$.
If $m$ is odd, then $T_{\ell+m-1}\subseteq E(\overrightarrow{G})$
since $\overrightarrow{G}$ is closed and \eqref{eq:il3}.
Therefore, $e_{\ell+m-1}\in E(\overrightarrow{G})$, and this contradicts the minimality of $m$. 
If $m$ is even, then by \eqref{eq:il1} and \eqref{eq:il3}, we have 
$\{e_{\ell-1},f_l,f_{\ell+2},\dots,f_{\ell+m-1},e_{\ell+m}\}\subseteq E(G)$.
Then, by Lemma~\ref{lem:pn}, we have $e_{\ell+m-1}\in E(\overrightarrow{G})$, again contradicting the minimality of $m$.
Therefore, \eqref{eq:tn2} holds.
Similarly, we can prove \eqref{eq:sn2}.

Let $K=\{v_{-2s},v_{-2s+2},\dots,v_0,v_1,\dots,v_\ell,v_{\ell+2},\dots,v_{\ell+2t}\}$.
If $K\neq\mathbb{Z}_n$,
then by \eqref{eq:il12}, \eqref{eq:il31}, \eqref{eq:tn2} and \eqref{eq:sn2}, $\overrightarrow{G}$ is not weakly connected.
This is a contradiction. 
Therefore, $K=\mathbb{Z}_n$, and in particular,  $s+\ell+1+t=|K|=n$.
From \eqref{eq:tn2} and \eqref{eq:sn2}, $s=t=\frac{n-\ell-1}{2}$.
Then, from \eqref{eq:il2}--\eqref{eq:tn2}, we have
\begin{align*}
\{f_{\ell+1},f_{\ell+3},\dots,f_{n-2}\}&\subseteq E(\overrightarrow{G}),
\\
\{f_\ell,f_{\ell+2},\dots,f_{n-3}\}&\subseteq E(\overrightarrow{G}),
\\
e_\ell,e_{\ell+1},\dots,e_{n-1}&\notin E(\overrightarrow{G}),
\end{align*}
respectively.
Together with
\eqref{eq:il1} and \eqref{eq:il11}--\eqref{eq:il31},
these imply
$E(\overrightarrow{G})=E(\overrightarrow{S_{n,\frac{n-\ell-1}{2},\ell-1}})$.
This proves $\overrightarrow{G}=\overrightarrow{S_{n,\frac{n-\ell-1}{2},\ell-1}}$.
\end{proof}

The following theorem is obtained immediately from Lemmas \ref{lem:2-3} and \ref{lem:2-2}.

\begin{thm} \label{th:4-4}
Let $\overrightarrow{G}$ be a nontrivial WCSC subgraph of $\overrightarrow{C_n^2}$.
Then 
there exist integers $j,k$ with $0\leq j\leq n-1$, $0\leq k\leq\lceil\frac{n-2}{2}\rceil$ such that $\overrightarrow{G}=\overrightarrow{S_{n,k,j}}$.
\end{thm}


\section{Enumerating directed spanning trees}\label{sec:4}

In this section we prove our main result which states that every directed spanning tree rooted by $v_j\in V(\overrightarrow{C_n^2})$ ($j\in\mathbb{Z}$) of $\overrightarrow{C_n^2}$ is contained in a unique WCSC subgraph.
Our method is a combinatorial formulation as in \cite{MT}. 
The tool we use is the homotopy theory for directed graphs~\cite{GLMY2, Lewis}.

\subsection{$C$-homotopy of directed graph}

First, we give some properties of the $C$-homotopy of a directed graph.
The $C$-homotopy group of the directed graph $\overrightarrow{G}$ with respect to a base vertex $v_j$ can be defined as given by Lewis~\cite{Lewis}.

We denote a path of length $\ell$ from $u$ to $v$ in a directed graph $\overrightarrow{G}$ by $p:=(u=u_0,u_1,\ldots,u_\ell=v)$ $(u_i\in V(\overrightarrow{G}),\;0\leq i\leq \ell)$, where $(u_{i-1},u_i)\in E(\overrightarrow{G})$ or $(u_i,u_{i-1})\in E(\overrightarrow{G})$.
If $p$ is given as above, then $p^{-1}:=(v=u_\ell,u_{\ell-1},\ldots,u_0=u)$.
We say that $p$ is a loop if $u$ and $v$ are the same, and we call $u$ the base vertex.
For each $u\in V(\overrightarrow{G})$, let $\psi(\overrightarrow{G},u)$ be the set containing all loops in $\overrightarrow{G}$ that have $u$ as their reference vertex.

We consider a binary relation $\sim$ on $\psi(\overrightarrow{G},u)$.
Let $p=(u=u_0,u_1,\dots,u_\ell=u)\in\psi(\overrightarrow{G},u)$.
Suppose there exists an integer $i$ ($0 \leq i \leq \ell$) and a vertex $v \in V(\overrightarrow{G})$ such that
\[
q=(u=u_0,u_1,\dots,u_{i-1},u_i,v,u_i,u_{i+1},\dots,u_\ell=u).
\]
Then we say that $q\in\psi(\overrightarrow{G}, u)$ extends $p$.
If $q$ extends $p$, then its length is $\ell+2$.
On the other hand, if $p$ does not extend $q$ for all $q \in\psi(\overrightarrow{G},u)$, then we say that $p$ is reduced.
Furthermore, if $\ell=0$ or if $p$ is reduced with $u_1 \neq u_{\ell-1}$ then we say that $p$ is cyclically reduced.
Let $S = (p_0, p_1, \dots, p_n)$ be a sequence of paths in $\psi(\overrightarrow{G},u)$.
We say that $S$ is admissible if for every $i$ ($1 \leq i \leq n$), either $p_{i-1}$ extends $p_i$ or $p_i$ extends $p_{i-1}$.
For any $p,q\in\psi(\overrightarrow{G},u)$, if there exists a nonnegative integer $n$ and paths $p=p_0,p_1,\dots,p_n=q\in \psi(\overrightarrow{G},u)$ such that the sequence $(p_0,p_1,\dots,p_n)$ is admissible, then we write $p \sim q$.
This equivalence relation $\sim$ is called \emph{$C$-homotopy}, and when $p \sim q$ we say that $p$ and $q$ are \emph{$C$-homotopic}.

Roughly speaking, the $C$-homotopy group $\pi_1(\overrightarrow{G},u)$ 
is the group formed by equivalence classes of based loops on a graph $\overrightarrow{G}$ with a fixed base vertex $u$.
Let $\pi_1(\overrightarrow{G},u)$ denote the set of equivalence classes of $\psi(\overrightarrow{G}, u)$ under this homotopy relation.
For any $p\in\psi(\overrightarrow{C_n^2},u)$, denote by $[p]$ the element of $\pi_1(\overrightarrow{G},u)$ containing $p$.
In particular, the concatenation on $\psi(\overrightarrow{G},u)$ induces a group structure on $\pi_1(\overrightarrow{G}, u)$; this group is called the \emph{$C$-homotopy group with respect to $u$}.

We refer the reader to \cite{GLMY2} for the precise definition of the $C$-homotopy group.

\subsection{Number of spanning tree in directed square cycles}

Moreover, the subgroup $\pi_1(\overrightarrow{G},u,3)$ can also be defined as in \cite{Lewis}.
For any $u \in \pi_1(\overrightarrow{G},u)$, $u$ has exactly one reduced representative, which is the unique representative of $u$ of minimal length.
Denote this representative by $\tilde{u}$ and write
\[
\tilde{u}=pqp^{-1}.
\]
Then the essential length of $u$ is defined as the length of $q$, where $q$ is cyclically reduced.
Let $\pi_1(\overrightarrow{G},u,3)$ denote the subgroup of $\pi_1(\overrightarrow{G},u)$ generated by elements whose essential length is at most $3$.
Namely, the subgroup $\pi_1(\overrightarrow{G},u,3)$ of $\pi_1(\overrightarrow{G},u)$ is generated by triangles.

It is clear that $\pi_1(\overrightarrow{G},u)=\pi_1(\overrightarrow{G},u,3)$ if $\overrightarrow{G}$ is a tree, directed strip graph, or directed strip graph with tails, while $\pi_1(\overrightarrow{G},u)\neq\pi(\overrightarrow{G},u,3)$ if $\overrightarrow{G}$ is a loop of length at least $4$ or $\overrightarrow{G}=\overrightarrow{C_n^2}$ with $n\geq7$.

Moreover, we say that $\bigcup_{i=0}^k A_i$ are disjoint if $A_i\cap A_{i'}=\emptyset$ for all $i,i'$ ($0\leq i<i'\leq k$).

\begin{thm} \label{th:gsnkj}
Let $n$ be an integer with $n\geq5$.
For every directed spanning tree $\overrightarrow{G}$ rooted by $v_j$ ($j\in\mathbb{Z}$) of $\overrightarrow{C_n^2}$,
there exists a unique nontrivial WCSC subgraph $\overrightarrow{H}$ of $\overrightarrow{C_n^2}$ such that $E(\overrightarrow{G})\subseteq E(\overrightarrow{H})$.
Each such graph $\overrightarrow{H}$ has a form $\overrightarrow{S_{n,k,j}}$ for $0\leq k \leq \lceil\frac{n-2}{2}\rceil$, $0\leq j\leq n-1$.
More precisely, for any $v_j\in V(\overrightarrow{C_n^2})$ ($j\in\mathbb{Z}$),
\begin{equation}\label{eq:dunion}
T_{\overrightarrow{C_n^2},v_j}=
\bigcup_{k=0}^{\lceil\frac{n-2}{2}\rceil}
T_{\overrightarrow{S_{n,k,j}},v_j} \quad \text{(disjoint).}
\end{equation}
\end{thm}

\begin{proof}
Since the assertion can be verified directly for $n=5$ and $6$, we assume $n\geq7$.
Clearly $\pi_1(\overrightarrow{G},v_j)$ is the trivial group for the directed spanning tree rooted by $v_j$ $\overrightarrow{G}$ of $\overrightarrow{C_n^2}$, so in particular $\pi_1(\overrightarrow{G},v_j)=\pi_1(\overrightarrow{G},v_j,3)$ holds.
A directed spanning tree rooted by $v_j$ $\overrightarrow{G}$ of $\overrightarrow{C_n^2}$ that is not closed with respect to some triangle $T_i$, $\pi_1(\overrightarrow{G'},v_j)=\pi_1(\overrightarrow{G'},v_j,3)$ also holds for the graph $\overrightarrow{G'}$ obtained from $\overrightarrow{G}$ by adding the unique missing edge of $T_i$.
This process can be iterated until we reach a closed subgraph containing $\overrightarrow{G}$.
The resulting graph $\overrightarrow{H}$ is a WCSC subgraph $\overrightarrow{H}$ of $\overrightarrow{C_n^2}$, and hence it is one of the graphs classified in Theorem~\ref{th:4-4}, or one of the trivial 
WCSC subgraphs.
Since \(\pi_1(\overrightarrow{H},v_j)=\pi_1(\overrightarrow{H},v_j,3)\)
holds only for nontrivial WCSC subgraph $\overrightarrow{H}$, there exist $j,k$ with $0\leq j\leq n-1$, $0\leq k\leq\lceil\frac{n-2}{2}\rceil$ such that $E(\overrightarrow{G})\subseteq E(\overrightarrow{S_{n,k,j}})$.

It remains to show that the union in \eqref{eq:dunion} is disjoint.
Suppose $E(\overrightarrow{G})\subseteq E(\overrightarrow{S_{n,k',j'}})$ for some $j',k'$ with $0\leq k'\leq \lceil\frac{n-2}{2}\rceil,\;0\leq j'\leq n-1$.
Then the subgraph with edge set $E(\overrightarrow{S_{n,k,j}})\cap \overrightarrow{E(S_{n,k',j'}})$ is a nontrivial WCSC subgraph of $\overrightarrow{C_n^2}$, and hence it coincides with $\overrightarrow{S_{n,k'',j''}}$ for some $j'',k''$ with $0\leq k''\leq \lceil\frac{n-2}{2}\rceil,\;0\leq j''\leq n-1$.
This implies $E(\overrightarrow{S_{n,k'',j''}})\subseteq E(\overrightarrow{S_{n,k,j}})$, which is possible only when $(j,k)=(j'',k'')$.
Then we have $(j,k)=(j',k')$.
Therefore,  the union in \eqref{eq:dunion} is disjoint.
\end{proof}

\begin{thm}
For $v_j\in\mathbb{Z}_n$ $(j\in\mathbb{Z})$,
\[
t(\overrightarrow{C_n^2},v_j)=J_n.
\]
\end{thm}

\begin{proof}
\begin{align*}
t(\overrightarrow{C_n^2},v_j)
&=\sum_{k=0}^{\lceil\frac{n-2}{2}\rceil}t(\overrightarrow{S_{n,k,j}},v_j) &&\text{(by Theorem \ref{th:gsnkj})}
\displaybreak[0]\\
&=\sum_{k=0}^{\lceil\frac{n-2}{2}\rceil}t(\overrightarrow{S_{n-2k}},1) &&\text{(by Lemma \ref{lem:strip_tail})}
\displaybreak[0]\\
&=\begin{cases}
\sum_{k=0}^{(n-2)/2}t(\overrightarrow{S_{2k+2}},1) &\text{if $n$ is even,}\\
1+\sum_{k=0}^{(n-1)/2}t(\overrightarrow{S_{2k+3}},1) &\text{if $n$ is odd}
\end{cases}
\displaybreak[0]\\
&=\begin{cases}
\sum_{k=0}^{(n-2)/2}(J_{2k}+J_{2k+1}) &\text{if $n$ is even,}\\
1+\sum_{k=0}^{(n-3)/2}(J_{2k+1}+J_{2k+2}) &\text{if $n$ is odd}
\end{cases} &&\text{(by Lemma \ref{lem:1})}
\displaybreak[0]\\
&=\begin{cases}
\sum_{k=0}^{n-1}J_k &\text{if $n$ is even,}\\
1+\sum_{k=1}^{n-1}J_k &\text{if $n$ is odd}
\end{cases}
\displaybreak[0]\\
&=J_n &&\text{(by Lemma \ref{jac3}).}
\end{align*}
\end{proof}

\begin{rem*}
We have verified by computer that the assertion of
Theorem \ref{th:gsnkj} holds for the directed graphs $\overrightarrow{C_9^3}$ and $\overrightarrow{C_{10}^3}$ (see \cite{WF, Yong}),
if we modify the definition of trivial weakly connected spanning subgraphs
to be the ones whose $C$-homotopy group is nontrivial.
There are exactly $61$ and $104$ nontrivial WCSC subgraphs up to automorphism of
$\overrightarrow{C_9^3}$ and $\overrightarrow{C_{10}^3}$, respectively, 
and every directed spanning tree is contained in a unique
nontrivial WCSC subgraph.
\end{rem*}

\section*{Acknowledgements}

This work was supported by JSPS Grant-in-Aid for JSPS Fellows (23KJ2020).


\begin{thebibliography}{9}

\bibitem{Baron} G. Baron, H. Prodinger, R. F. Tichy, F. T. Boesch, J. F. Wang, The number of spanning trees in the square of a cycle, {\sl Fibonacci Quart.} {\bf23}, no. 3 (1985), 258--264.

\bibitem{Chen} X.~Chen. : The number of spanning trees in directed circulant graphs with non-fixed jumps,
\textit{Discrete Math.} \textbf{307} (2007) 1873--1880.

\bibitem{GR} C. Godsil and G. Royle, Algebraic Graph Theory, Springer, 2001.

\bibitem{GLMY1} A.~Grigor'yan, Y.~Lin, Y.~Muranov, S.~-T.~Yau. : Homologies of path complexes and digraphs, arXiv: 1207.2834v4 (2013).

\bibitem{GLMY2} A.~Grigor'yan, Y.~Lin, Y.~Muranov, S.~-T.~Yau. : Homotopy theory for digraphs, {\sl Pure Appl. Math. Q.} {\bf10} (2014), 619--674.

\bibitem{Harary} F.~Harary, Graph Theory, Addison-Wesley, 1969

\bibitem{AFH} A.~F.~Horadam. : Jacobsthal representation numbers, \textit{Fibonacci Quart.}, 34.1 (1996) 68--74.

\bibitem{Kleitman}  D. J. Kleitman, B. Golden. Counting trees in a certain class of graphs, {\sl The Amer. Math. Monthly} {\bf82}, no. 1, (1975), 40--44.

\bibitem{Lee} P.~D.~Leenheer, An elementary proof of a matrix tree theorem for directed graphs, \textit{Society for industrial and Applied Mathematics}, \textbf{62} No.3 (2020) pp.716--726.

\bibitem{Lewis} H. A. Lewis. Homotopy in $Q$-polynomial distance-regular graphs, {\sl Discrete Math.} {\bf223} (2000), 189--206.

\bibitem{Louis} J.~Louis, Spanning trees in directed circulant graphs and cycle power graphs, \textit{Monatshefte f\"{u}r Mathematik}, \textbf{182} (2017) 51--63.

\bibitem{MY} F.~Ma, B.~Yao. : An iteration method for computing the total number of spanning trees and its applications in graph theory, \textit{Theoretical Computer Science} \textbf{708} (2018) 46--57.

\bibitem{Med2} A. D. Mednykh, I. A. Mednykh. The number of spanning trees in circulant graphs, its arithmetic properties and asymptotic, {\sl Discrete Math.} {\bf342}, no.6 (2019), 1772--1781.

\bibitem{Med} A. D. Mednykh, I. A. Mednykh. On rationality of generating function for the number of spanning trees in circulant graphs, {\sl Algebra Colloq.} {\bf27}, no.1 (2020), 87--94.

\bibitem{MT} A.~Munemasa, Y.~Tanaka. : Convex subgraphs and spanning trees of the square cycles, \textit{The Australasian Journal of Combinatorics}, Volume 88(2) (2024), 204--211.

\bibitem{Tanaka} Y.~Tanaka. : On the average hitting times of Cay($\mathbb{Z}_N,\{+1,+2\}$),  {\sl Discrete Appl. Math.}, 343 (2024), Pages 269--276.

\bibitem{Tutte} W.~T.~Tutte. : \textit{Graph Theory}, Cambridge University Press, 2001.



\bibitem{WF} J.~M.~Wojciechowski, M.~Fellows. : Counting spanning trees in directed regular multigraphs, \textit{Journal of the Franklin Institute}, \textbf{326}, (1989), 889--896.

\bibitem{Yong} X. Yong, T. Acenjian. The numbers of spanning trees of the cubic cycle $C_N^3$ and the quadruple cycle $C_N^4$, {\sl Discrete Math.} {\bf169} (1997), 293--298.

\bibitem{ZYG} Y.~Zhang, X.~Yong, M.~J.~Golin. The numbers of spanning trees of the sirculant graphs, {\sl Discrete Math.} {\bf223} (2000), 337--350.


\end{thebibliography}
\end{document}